\documentclass[preprint,12pt]{elsarticle}
 
\usepackage{etoolbox}
\makeatletter
\def\ps@pprintTitle{%
	\let\@oddhead\@empty
	\let\@evenhead\@empty
	\def\@oddfoot{\reset@font\hfil\thepage\hfil}
	\let\@evenfoot\@oddfoot
}
\makeatother
\usepackage{natbib}
\usepackage{bm}
\usepackage{float}
\usepackage{graphics}
\usepackage{amsmath, amsthm, amssymb}
\usepackage{caption}
\usepackage{xcolor}
\usepackage{hyperref}

\hypersetup{
    colorlinks=true,
    linkcolor=blue
}
\usepackage{lscape}
\usepackage{float}
\usepackage[utf8]{inputenc}
\usepackage{amsthm}

\usepackage{enumerate}
\usepackage{mathrsfs}
\usepackage{placeins}
\usepackage[inner=4cm,outer=2cm]{geometry}
\usepackage{graphicx}
\usepackage{amsfonts}
\usepackage{verbatim}
\usepackage{amssymb}
\usepackage{amsthm,multirow}
\usepackage{orcidlink}

\newtheorem{thm}{Theorem}[section]
\newtheorem{lem}{Lemma}[section]

\newtheorem{prop}{Proposition}[section]

\theoremstyle{definition}
\newtheorem{defn}{Definition}[section]
\newtheorem{ex}{Example}[section]
\newtheorem{res}{Result}[section]

\theoremstyle{remark}
\newtheorem{rem}{Remark}[section]

\theoremstyle{properties}

\theoremstyle{definition}

\numberwithin{equation}{section}

\usepackage[mathscr]{euscript}
\usepackage{geometry} 
\usepackage{graphicx,lscape} 
\usepackage{booktabs} 
\usepackage{bigstrut}
\usepackage{array}
\usepackage{paralist} 
\usepackage{verbatim}
\usepackage{subfig} 
\biboptions{comma,round,authoryear}
\begin{document}
\begin{frontmatter}
	\title{\textbf{Fractional Cumulative Past Inaccuracy \\in the Quantile Framework and its Applications}}
	\author{Iona Ann Sebastian \orcidlink{0009-0000-7640-2215}\corref{cor1}}
    \author{S. M. Sunoj \orcidlink{0000-0002-6227-1506}}
	\ead{ionaann99@gmail.com, smsunoj@cusat.ac.in}
	\cortext[cor1]{Corresponding author}

	\address{Department of Statistics\\Cochin University of Science and Technology\\Cochin 682 022, Kerala, INDIA}

	\begin{abstract}
Fraction-based information measures have received considerable attention for describing complex systems, as they enable the investigation of signals with high sensitivity \citep{machado2014fractional}. In this paper, we introduce quantile versions of fractional cumulative past inaccuracy (FCPI) and dynamic fractional cumulative past inaccuracy (DFCPI) measures based on the inverse Mittag-Leffler function (MLF) or fractional logarithm function, which are the extensions of the fractional cumulative past and dynamic fractional cumulative past entropies, respectively. The various properties of quantile-based FCPI and its dynamic version are provided. We also propose a
nonparametric estimator for the proposed measure, and simulation studies are
carried out for validation. Finally, we bring out real data application of the newly introduced quantile-based FCPI.
	\end{abstract}
	

\end{frontmatter}

\section{Introduction}

Entropy and inaccuracy measures provide fundamental tools for quantifying uncertainty and discrepancy between probability distributions and have found wide applications in information theory, probability, statistics, reliability, and survival analysis. The classical Shannon entropy is based on the logarithmic information function and, for a discrete random variable (rv) $X$ with probability mass function $\{p_i\}_{i=1}^n$, is given by $H(X)=-\sum_{i=1}^{n}p_i\log p_i$.  The need for greater flexibility in measuring uncertainty has led to several generalized entropy families, in which an additional parameter controls the contribution of different regions of the underlying probability distribution. A further generalization arises from fractional calculus, where the order of the underlying operator is allowed to be non-integer.  In fractional calculus, ordinary derivatives such as $\frac{d}{dx}$ are generalized to derivatives of non-integer order, $\frac{d^{\alpha}}{dx^{\alpha}}$, where $\alpha$ is the fractional order.  This idea has led researchers to construct entropy-like quantities involving fractional derivatives or fractional integrals.  For a discrete rv $X$, the fraction entropy of order $\alpha$, is given by $S_{\alpha}(P) = \sum_{i=1}^{n}{p_i \left(-\log p_i\right)^{\alpha}}, \; 0 < \alpha < 1$, developed based on the fractional calculus \citep{ubriaco2009entropies}.  For $0 < \alpha < 1$, small probabilities receive relatively greater weight than they do under larger values of $\alpha$, and hence fractional entropy captures uncertainty contained in the rv more effectively to rare or low-probability events.  This framework has proved useful in describing nonlocality, memory, anomalous diffusion, and complex dynamical phenomena, and has consequently motivated the development of fractional information and entropy measures (see \citet{liang2022computation}).  For some recent works on fractional entropy, their properties and applications, see \cite{di2019past}, \cite{saha2023extended}, and \cite{foroghi2023extensions} and references therein.\\

The Mittag--Leffler function (MLF) \citep{mittag1903nouvelle}, plays a central role in fractional calculus. For $\alpha>0$, the one-parameter MLF is defined by
\[
E_{\alpha}(x)
=
\sum_{k=0}^{\infty}\frac{x^k}{\Gamma(1+\alpha k)},
\]
which reduces to the exponential function when $\alpha=1$, where $\alpha !=\Gamma(1+\alpha)$, and $\Gamma(\cdot)$ denotes complete gamma function. The occurrence of $E_{\alpha}$ in solutions of fractional differential equations provides a natural mathematical basis for introducing fractional-order analogues of the exponential and logarithmic functions. In particular, \citet{jumarie2012derivation} established a connection between fractional information theory and the inverse MLF. If ${Ln}_{\alpha}$ denotes the inverse of $E_{\alpha}$, $i.e., {Ln}_{\alpha} (x) = E_{\alpha}^{-1} (x)$ then ${Ln}_{\alpha}$ may be interpreted as a logarithm of fractional order.  Various properties of the inverse MLF \citep{jumarie2012derivation} include, (i) $Ln_\alpha 1=0$, $Ln_\alpha 0=-\infty$, $Ln_\alpha x <0$, when $x<1$; (ii) $1(Ln_\alpha 1)^\frac{1}{\alpha}=0=0(Ln_\alpha 0)^\frac{1}{\alpha}$; (iii) $\frac{d^\alpha}{dx^\alpha}(Ln_\alpha x)^\frac{1}{\alpha}=\frac{\alpha !}{((1-\alpha)!)^2}\frac{1}{x^\alpha}$; (iv) $Ln_\alpha(x^a)=a^\alpha Ln_\alpha x$; (v) $[Ln_\alpha(mn)]^\frac{1}{\alpha}=[Ln_\alpha(m)]^\frac{1}{\alpha}+[Ln_\alpha(n)]^\frac{1}{\alpha}$. The fractional logarithmic function, denoted by $Ln_\alpha(\cdot)$, has no closed form. \citet{jumarie2012derivation} used this inverse to construct a fractional logarithmic operator and, consequently, generalized entropy measures. In this framework, the classical logarithm is recovered at the limiting order $\alpha=1$, since  ${Ln}_{1}(x) = E_1^{-1}(x) = \log x$.  Thus, the replacement $\log x \longrightarrow {Ln}_{\alpha}(x)$ provides a mathematical connection from classical information measures to fractional-order analogues. \citet{sarumi2020highly} developed highly accurate global Pad\'e approximations for the generalized two-parameter MLF and discussed their use in approximating its inverse.  \\

For a discrete distribution, Jumarie's fractional entropy of order $\alpha$ can be expressed in the form $\widetilde H_{\alpha}(X)
= -\sum_{i=1}^{n} p_i\left[{Ln}_{\alpha}(p_i)\right]^{1/\alpha}, \; 0<\alpha<1$.  However, this formulation need not be non-negative. To overcome this problem, \citet{zhang2020cumulative} introduced a modified fractional entropy based on the inverse MLF, $H_{\alpha}(X) = \sum_{i=1}^{n} p_i \left[-{Ln}_{\alpha}(p_i)\right]^{1/\alpha}$, $0<\alpha<1$, which is non-negative and retains a meaningful fractional-order interpretation. An analgous representation of $H_{\alpha}$  for a continuous rv $X$, with probability density function (PDF) $f(\cdot)$, given by 
\begin{equation}
	H_{\alpha}(X) = \int_{\mathbb{R}} {f(x) \left(-Ln_{\alpha} \left(f(x)\right)\right)^{\frac{1}{\alpha}} dx}, \; 0 < \alpha < 1.
\end{equation}
Motivated by such a construction of the fractional entropy, \cite{saha2023extended} proposed an information measure that can provide better information than the cumulative entropy using the power of the inverse MLF, known as the extended fractional cumulative past entropy (EFCPE). The EFCPE for a rv $X$ is defined as
\begin{equation}\label{EFCPE}
\mathcal{\overline{CE}_\alpha}(X)=\int_{0}^{\infty} F(x)[-Ln_\alpha F(x)]^\frac{1}{\alpha} dx.
\end{equation}
Inaccuracy measures provide a related but distinct perspective by quantifying the discrepancy between two distributions. For two nonnegative and absolutely continuous random variables $X$ and $Y$ with PDF's $f(\cdot)$ and $g(\cdot)$, Kerridge's inaccuracy measure (\citet{kerridge1961inaccuracy}, \citet{nath1968entropy}) is given by, $I(X, Y) = -\int_{0}^{\infty} f(x)\log g(x)\,dx$, with Shannon differential entropy $H(X) = -\int_{0}^{\infty}{f(x) \log f(x) dx}$ obtained as the special case $g=f$. Cumulative versions of inaccuracy measures and their applications are well studied in literature. A complementary development in entropy theory is by representing the entropy through the distribution function (DF) $F(\cdot)$ or survival function (SF) $\bar{F}(\cdot) = 1 - F(\cdot)$ and, consequently, does not require explicit specification or estimation of the density.  Based on this idea, \citet{kharazmi2024fractional} introduced fractional cumulative residual inaccuracy information and related Jensen-type measures, thereby illustrating the application of fractional cumulative inaccuracy measures in reliability, survival analysis, and dynamical systems.  Associated with past lifetime or inactivity information, \citet{saha2025fractional} introduced a fractional cumulative past inaccuracy (FCPI) measure and its dynamic version for lifetime distributions.  Incorporating the inverse MLF, the FCPI of $X$ and $Y$ is defined by
\begin{equation}\label{FCPI}
	\mathcal{CI_\alpha}(X,Y)= \int_{0}^{\infty}{F}(x)[-Ln_{\alpha}(G(x))]^\frac{1}{\alpha} dx.
\end{equation}
When $X$ represents the actual distribution and $Y$ a reference, fitted, or model distribution, FCPI measures the information discrepancy resulting from using $G$ when the actual distribution is $F$ and therefore useful in evaluating the model misspecification or distributional discrepancy.  The potential applications of \eqref{FCPI} are highlighted in various fields such as statistics, information theory, image processing, and machine learning. \\ 

Although many research on the fractional information measures are available in literature, however, these formulations are predominantly in terms of the DF, SF, or PDF. On the other hand, the quantile function (QF), defined by
\begin{equation*}
	Q(u) = \inf \{x|F(x)\geq u\}, \; 0 \leq u \leq 1
\end{equation*}
is an efficient and equivalent alternative (see \cite{gilchrist2000statistical}) in modelling and analysis of statistical data and offers several advantages in modelling and analysis of lifetime data.  Since $Q$ uniquely determines $F$, quantile-based information measures provide an equivalent but potentially more flexible description of distributional uncertainty. In particular, the transformation $x=Q(u), \; u=F(x)$, permits cumulative information functionals to be expressed directly on the probability scale $(0,1)$. Such a formulation is attractive as empirical quantiles provide a natural nonparametric basis for estimation and quantile representations are informative for asymmetric and heavy-tailed distributions. Some recent investigations on quantile-based fractional cumulative information measures, we refer to \citet{sebastian2025fractional}, \citet{paul2025quantile}, \citet{sunoj2026fractional} and the references therein.\\

Further, in many real-world situations, the uncertainty of a rv has to be measured in backward, known as the past lifetime.  For a non-negative lifetime rv $X$, the past lifetime (inactivity time) at age $t$ is $X^t = \left(t - X \mid X \leq t\right)$, describes how long ago failure occurred conditional on failure having occurred before $t$.  This makes past-lifetime information particularly relevant to retrospective failure analysis, inspection, maintenance and repair.  These observations motivate the development of a quantile-based fractional cumulative past inaccuracy measure. The principal contribution of the paper is therefore the integration of the Mittag--Leffler-based fractional information structure with a quantile formulation of cumulative past inaccuracy.  The fractional parameter $\alpha$ further introduces an additional degree of flexibility in assessing the discrepancy between the past-lifetime distributions.  This provides a unified framework in which the fractional order controls the information scale while the quantile representation captures the distributional characteristics of past lifetimes. \\

The remainder of the paper is organized as follows. Section~2 introduces the proposed quantile-based fractional cumulative past inaccuracy measure and establishes its fundamental properties. Section~3 develops its dynamic version and investigate its various properties. Section~4 discusses nonparametric estimation and associated asymptotic properties. Numerical illustrations and applications to lifetime data are presented in Section~5, followed by concluding remarks in Section~6.

\section{Quantile-based FCPI}

In this section, we study fractional cumulative past inaccuracy employing the concept of inverse MLF from a quantile point of view. This approach is an alternative method to the study of \eqref{FCPI}. Consider two absolutely continuous non-negative rvs, 
$X$ and $Y$ having CDFs $F(\cdot)$, $G(\cdot)$ and QFs $Q_X(\cdot)$, $Q_Y(\cdot)$. Further, $F(Q_X(u))=u$, $0<u<1$ and on differentiating $F(Q_X(u))=u$, we obtain
\begin{equation*}
q_X(u)f(Q_X(u))=1, \; 0<u<1.
\end{equation*}
Here, $q_X(u)$ denotes the quantile density function (QDF) and  $f(Q_X(u))$ denotes the density quantile function corresponding to $F$. Also, $Q_3(\cdot)=Q_Y^{-1}(Q_X(\cdot))=G(F^{-1}(\cdot))$ denote the QF of $F(G^{-1}(\cdot))$ and $q_3(u)=\frac{d}{du}Q_3(u)$ be the QDF of $Q_3(u)$. A useful reliability measure in the context of past lifetime rv is the reversed hazard function of $X$, given by $\tilde{h}_{X}(x)=\frac{f(x)}{F(x)}$. The corresponding quantile measure is the reversed hazard quantile function, 
\begin{equation}
\tilde{h}_{XQ}(u)=(u q_X(u))^{-1}.
\end{equation}
Another important reliability measure closely associated with the reversed hazard rate function is the mean past lifetime (MPL) of $X$ for $x>0$, defined as $\mathcal{M}_X(x)=E(x-X|X<x) = \frac{1}{F(x)}\int_0^x {F(u)du}$. Analogously to MPL, the corresponding function is the mean past quantile function \citep{nair2013quantile} of $X$, obtained as
\begin{equation}\label{Q-MPL}
\mathcal{M}_{XQ}(u)=u^{-1}\int_{0}^{u}pq_X(p)dp, \; 0<u<1.
\end{equation}
\begin{defn}
\textnormal Let $X$ and $Y$ have QFs $Q_X(\cdot)$ and $Q_Y(\cdot)$, respectively. Then the Q-FCPI measure of order $\frac{1}{\alpha}$ is
\begin{eqnarray}\label{qfcpi}
\mathcal{CI_\alpha^{Q}}(X,Y)&=& \int_{0}^{1}{F}(Q_X(p))[-Ln_{\alpha}(G(Q_X(p)))]^\frac{1}{\alpha} q_X(p)dp\\ 
{}&=&\int_{0}^{1}p[-Ln_{\alpha}(Q_3(p))]^\frac{1}{\alpha} q_X(p)dp\nonumber, \hspace{0.3cm}0<\alpha< 1.
\end{eqnarray}
\end{defn}
In \eqref{qfcpi}, the arguments of the fractional order logarithmic function is the QF, $Q_3(\cdot)$ defined on the interval [0,1], which assures the non-negativity of the proposed inaccuracy measure, that is $\mathcal{CI_\alpha^{Q}}(X,Y) \geq 0$. The measure \eqref{qfcpi} can be used to find the discrepancy between highly volatile components based on their past lifetimes, also useful in the analysis of chaotic systems where small perturbations can result in significant structural changes. In reliability studies, the past lifetime is the elapsed time between the failure of a system and the time when it is found to be down. The measure is an extension of the Q-EFCPE given in \eqref{FCPI}.  As $Ln_{\alpha}(\cdot)$ does not have a closed-form expression (see \cite{liang2018diffusion}), for the computation, the approximate expression of the Q-FCPI measure is followed, which can be deduced using $Ln_{\alpha} x \approx log \hspace{ 0.1cm} x^{\alpha!}$ for $0 < \alpha \leq 1$ (see \cite{jumarie2012derivation}).  Then
\begin{eqnarray}\label{aqfcpi}
\mathcal{CI_\alpha^{Q}}(X,Y)&\approx& \int_{0}^{1}[-\log(Q_3(p))^{\alpha !}]^\frac{1}{\alpha} p \ q_X(p)dp\\ 
{}&\approx& (\alpha !)^\frac{1}{\alpha}\int_{0}^{1}[-\log(Q_3(p))]^\frac{1}{\alpha} p \ q_X(p)dp\nonumber, \hspace{0.3cm}0<\alpha< 1.
\end{eqnarray}
Using the expansion of the Maclaurin series \citep{spiegel1953some},
\begin{equation*}
-\log(1-x)=\sum_{i=1}^{n} \frac{x^i}{i}, \; 0<x<1,
\end{equation*}
the Q-FCPI measure in \eqref{qfcpi} takes the form,
 \begin{equation*}
     \mathcal{CI_\alpha^Q}(X,Y)\approx(\alpha!)^\frac{1}{\alpha}\int_{0}^{1}\bigg(\sum_{i=1}^{\infty}\frac{(1-Q_3(p))^i}{i}\bigg)^{\frac{1}{\alpha}} p \ q_X(p)dp.
\end{equation*}

{

Using the approximation of the inverse MLF, $Ln_{\alpha} x \approx \log x^{\alpha !}$ for $0<\alpha \leq 1$, \cite{saha2025fractional} proposed a modified version of the FCPI measure (MFCPI) for two CDFs $F(\cdot)$ and $G(\cdot)$ as
\begin{equation}\label{MFCPI}
\mathcal{\widetilde{CI}_{\alpha}}(X,Y)=-\int_{0}^{\infty} F(x) Ln_{\alpha} G(x)dx \approx \alpha! \hspace{0.2cm}\mathcal{CI}(X,Y), \hspace{0.2cm} 0<\alpha<1.
\end{equation}
A quantile version of \eqref{MFCPI} (Q-MFCPI) as
\begin{equation}\label{qmfcpi}
\mathcal{\widetilde{CI}_{\alpha}^Q}(X,Y)=-\int_{0}^{1} p Ln_{\alpha} Q_3(p)q_{X}(p)dp \approx \alpha! \hspace{0.2cm}\mathcal{CI}^{Q}(X,Y), \hspace{0.2cm} 0<\alpha<1.
\end{equation}
Next, we present some important facts that can be easily deduced from \eqref{qfcpi}.
\begin{itemize}
\item The measures $\mathcal{CI_\alpha^{Q}}(X,Y)$ and $\mathcal{CI_\alpha^{Q}}(Y,X)$ are not symmetric.
\item \eqref{qfcpi} can be represented in terms of the reversed hazard quantile function of $X$, $\tilde{h}_{XQ}(u)=(u q_X(u))^{-1}$ as
\begin{equation}
\mathcal{CI_\alpha^{Q}}(X,Y)=\int_{0}^{1}[-Ln_{\alpha}(Q_3(p))]^\frac{1}{\alpha}[\tilde{h}_{XQ}(p)]^{-1}dp.
\end{equation}
\item For $\alpha=1$, the Q-FCPI in \eqref{qfcpi} reduces to the quantile-based cumulative past inaccuracy proposed by 
\cite{kayal2018quantile}.
\item If the two QFs $Q_X(\cdot)$ and $Q_Y(\cdot)$ coincide, $\mathcal{CI_\alpha^{Q}}(X,Y)$ reduces to the quantile-based fractional cumulative past entropy (see \cite{paul2025quantile}).
\end{itemize}
Now, we obtain closed form expression of the Q-FCPI measure using \eqref{aqfcpi} for some statistical models. 
\begin{ex}
Let $X$ and $Y$ follow power distribution with QFs $Q_X(u)=\beta u^{1/\gamma}$ and $Q_Y(u)=\beta u^{1/\delta}$, where $\beta,\gamma,\delta >0$  respectively. We obtain $q_X(u)=\frac{\beta}{\gamma}u^{\frac{1}{\gamma} -1}$ and $Q_3(u)=u^{\delta/\gamma}$. Then, the closed-form expression of Q-FCPI for $0<\alpha\leq1$ using \eqref{aqfcpi} is obtained as
$$\mathcal{CI_\alpha^{Q}}(X,Y)\approx\frac{\beta(\alpha! \delta)^{1/\alpha}\Gamma(1+\frac{1}{\alpha})}{(1+\gamma)^{1+1/\alpha}}.$$
and the plot of $\mathcal{CI_\alpha^{Q}}(X,Y)$ for different parameteric values are displayed in Figure \ref{fig1}.
\end{ex}
\begin{figure}[H]
\centering
\includegraphics[width=1.1\linewidth]{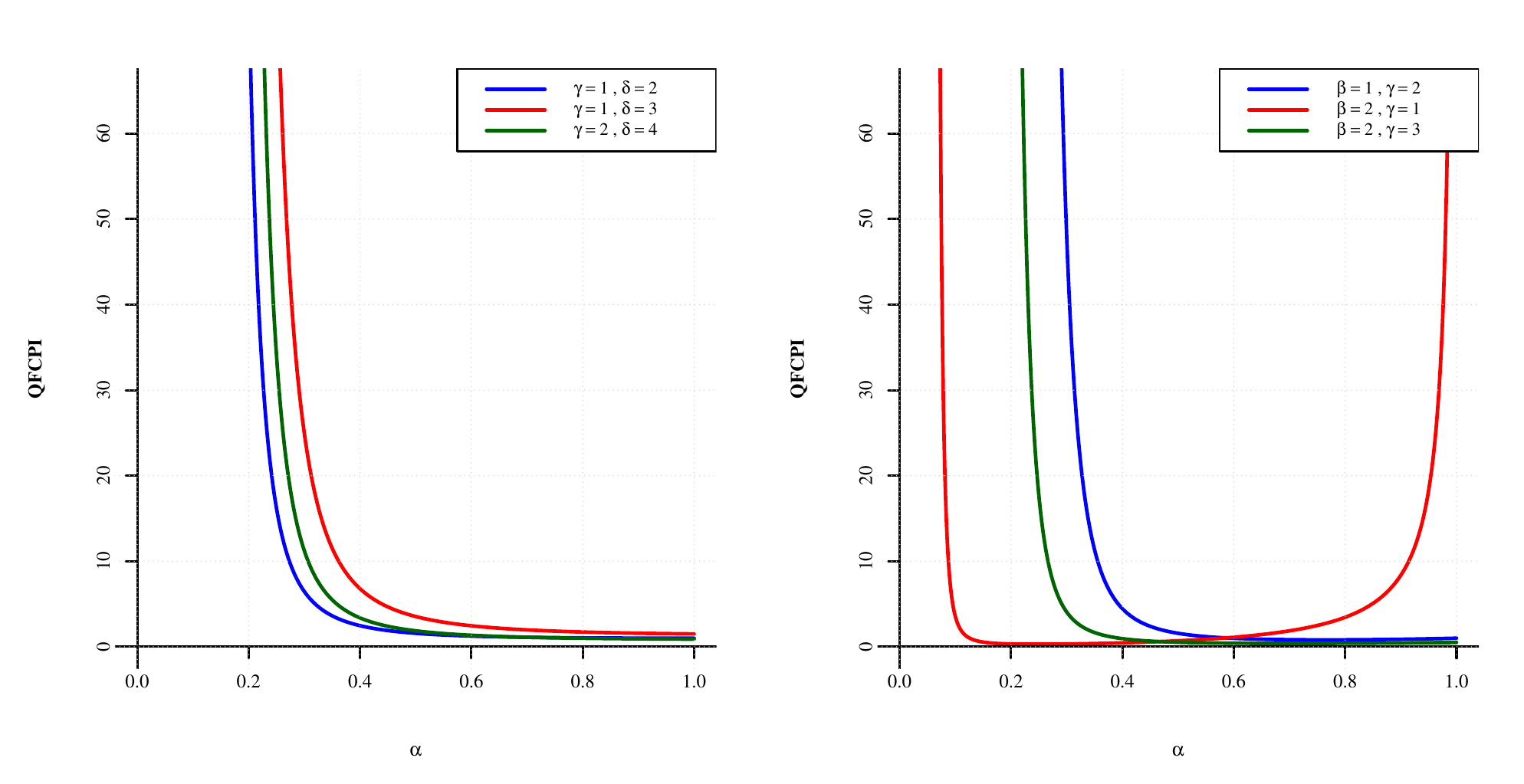}
\caption{Plot of Q-FCPI: For power distributions ($\beta=2$) in Example 2.1 and for Frechet distribution  in Example 2.2} 
\label{fig1}
\end{figure}
\begin{ex}
Consider the QFs of Frechet distributions $Q_X(u)=(-\log u)^{-1/\beta}$ and $Q_Y(u)=(-\log u)^{-1/\gamma}$, where $\beta,\gamma>0$. In this case, $Q_3(u)=e^{-(-\log u)^{\gamma/\beta}}$ and $q_X(u)=\frac{1}{\beta u}(-\log u)^{-(\frac{1}{\beta}+1)}$. By using \eqref{aqfcpi} the closed-form expression of Q-FCPI is given by
\[
\mathcal{CI_\alpha^{Q}}(X,Y)\approx
\begin{cases}
\frac{(\alpha!)^{1/\alpha}}{\beta}\Gamma \left(\frac{\gamma}{\beta \alpha}-\frac{1}{\beta}\right), & \mbox{if} \hspace{0.2cm} \gamma>\alpha,\\[0.5em]
+\infty,
& \mbox{if} \hspace{0.2cm} 0<\gamma<\alpha,
\end{cases}
\]
and the plot of $\mathcal{CI_\alpha^{Q}}(X,Y)$ for $\beta=0.8$, $\gamma=3$ is given in Figure \ref{fig2}.
\end{ex}
\begin{figure}[H]
\centering
\includegraphics[width=0.75\linewidth]{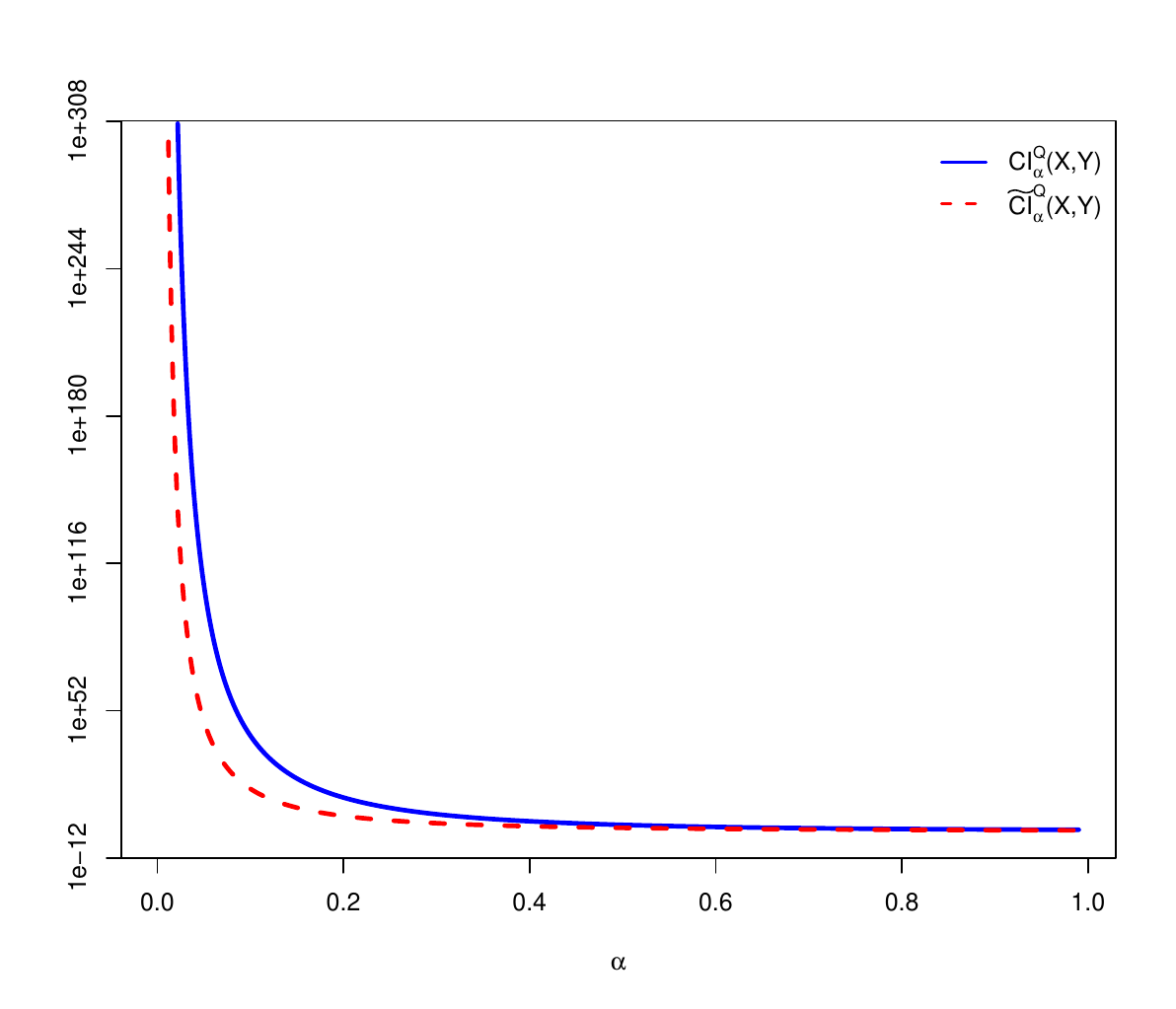}
\caption{Plot of $\mathcal{CI_\alpha^Q}(X,Y)$ and $\mathcal{\widetilde{CI}_\alpha^Q}(X,Y)$ for $\beta=0.8$, $\gamma=3$ in Example 2.2} 
\label{fig2}
\end{figure}
Next, we consider some examples to illustrate the usefulness of the quantile-based tool in \eqref{qfcpi} for computing the inaccuracy between $X$ and $Y$, where $X$ does not have an explicit form of CDF, though it has a closed-form of QF.
\begin{ex}
Let $X$ and $Y$ have QFs $Q_X(u)=2u-u^2$ and $Q_Y(u)=u$, $0<u<1$ respectively. Note that $Q_X(u)$ is the QF of a special case of Govindarajulu distribution and $Q_Y(u)$ is the QF of uniform distribution. We obtain $Q_3(u)=2u-u^2$ and $q_X(u)=2-2u$. By using \eqref{aqfcpi} an expression of Q-FCPI is given by
$$\mathcal{CI_\alpha^{Q}}(X,Y)\approx  (\alpha !)^\frac{1}{\alpha}\int_{0}^{1}(2p-2p^2)[-\log(2p-p^2)]^\frac{1}{\alpha}dp.$$
This integral can be obtained numerically for particular values of $\alpha$. For example, when $\alpha=0.2$, $\mathcal{CI_\alpha^{Q}}(X,Y)\approx 0.631$ and for $\alpha=0.5$, $\mathcal{CI_\alpha^{Q}}(X,Y)\approx 0.108$.
\end{ex}
\begin{ex}
Consider the QFs of a generalized lambda distribution $Q_X(u)=\frac{1}{\lambda_1}+\frac{1}{\lambda_2}(u^{\lambda_3}-(1-u)^{\lambda_4})$ and a uniform distribution $Q_Y(u)=u$ respectively, where $\lambda_1,\lambda_2,\lambda_4>0$ and $\lambda_3$ is a positive integer. Further, it is known that $Q_X(u)$ represents a life distribution if $\lambda_1-\lambda_2^{-1} \geq 0$. For simplicity of our calculation, we consider $\lambda_1=\lambda_2=1$. In this case, $Q_3(u)=1+u^{\lambda_3}-(1-u)^{\lambda_4}$ and $q_X(u)=\lambda_3u^{\lambda_3-1}+\lambda_4(1-u)^{\lambda_4-1}$. Substituting these in \eqref{aqfcpi}, we obtain
$$\mathcal{CI_\alpha^{Q}}(X,Y)\approx  (\alpha !)^\frac{1}{\alpha}\int_{0}^{1}p[-\log(1+p^{\lambda_3}-(1-p)^{\lambda_4})]^\frac{1}{\alpha}(\lambda_3p^{\lambda_3-1}+\lambda_4(1-p)^{\lambda_4-1})dp.$$
which can be computed numerically for specific values of the parameters. For instance, when $\lambda_3=1,\lambda_4=0.2,\alpha=0.2$, we have $\mathcal{CI_\alpha^{Q}}(X,Y)\approx0.9923$ and when $\lambda_3=1,\lambda_4=0.2,\alpha=0.5$, we get $\mathcal{CI_\alpha^{Q}}(X,Y)\approx0.3063$.
\end{ex}
\begin{ex}
(i) Let $X$ and $Y$ be two non-negative and absolutely continuous rvs with QFs $Q_X(\cdot)$ and $Q_Y(\cdot)$, respectively, which satisfy PHM (proportional hazards model), 
\begin{equation}\label{phm}
Q_Y(u)=Q_X(1-(1-u)^\frac{1}{\theta}),    
\end{equation}
where $\theta>0$. In this case, $Q_3(u)=1-(1-u)^\theta$. Thus, from \eqref{aqfcpi} we get $$\mathcal{CI_\alpha^{Q}}(X,Y)\approx (\alpha !)^\frac{1}{\alpha}\int_{0}^{1}p[-\log(1-(1-p)^\theta)]^\frac{1}{\alpha}q_X(p)dp.$$\\
(ii) Assuming $X$ and $Y$ satisfy PRHM (proportional reversed hazards model), that is,
\begin{equation}\label{PRHM}
Q_Y(u)=Q_X(u^{1/\theta}), 
\end{equation}
where $\theta$ is a positive integer. Here, $Q_3(u)=u^\theta$, then $$\mathcal{CI_\alpha^{Q}}(X,Y)\approx (\alpha !\theta)^\frac{1}{\alpha}\int_{0}^{1}p[-\log p]^\frac{1}{\alpha}q_X(p)dp=(\alpha !\theta)^\frac{1}{\alpha} \mathcal{E_\alpha^Q}(X),$$
where $\mathcal{E_\alpha^Q}(X)$ is the quantile-based fractional cumulative past entropy (Q-FCPE) of $X$.
\end{ex}
\begin{res}
Assuming $X$ and $Y$ are rvs with respective QFs $Q_X(\cdot)$ and $Q_Y(\cdot)$. For $0<\alpha<1$ and $0\leq q_X(u)\leq 1$, the Q-FCPI and Q-MFCPI satisfy the relationship
$$\mathcal{CI_\alpha^{Q}}(X,Y) \geq \left[\mathcal{\widetilde{CI}_{\alpha}^Q}(X,Y)\right]^\frac{1}{\alpha}. $$
\end{res}
\begin{rem}
Analytically, comparing $\mathcal{CI_\alpha^Q}(X,Y)$ with $\mathcal{\widetilde{CI}_{\alpha}^Q}(X,Y)$ is generally difficult and therefore we often employ graphical methods. Using Example 2.2, with $\mathcal{CI_\alpha^Q}(X,Y) \approx
\frac{(\alpha!)^{1/\alpha}}{\beta}\Gamma \left(\frac{\gamma}{\beta \alpha}-\frac{1}{\beta}\right)$ and $\mathcal{\widetilde{CI}_{\alpha}^Q}(X,Y) \approx \frac{(\alpha!)}{\beta}\Gamma \left(\frac{\gamma-1}{\beta}\right)$ we illustrate this. The plots of $\mathcal{CI_\alpha^Q}(X,Y)$ and $\mathcal{\widetilde{CI}_{\alpha}^Q}(X,Y)$ for $\beta=0.8$ and $\gamma=3$ are given in Figure \ref{fig2}, validating Result 2.1.
\end{rem}
\begin{thm}
Suppose $X$ and $Y$ are two non-negative rvs. Then
    \begin{enumerate}
\item $\mathcal{CI_{\alpha}^{Q}}(X,Y) \geq (\alpha !)^{\frac{1}{\alpha}} \int_{0}^{1} p(Q_3(p)-1)^\frac{1}{\alpha} q_X(p) dp$, 
\item $\mathcal{CI_{\alpha}^{Q}}(X,Y)\geq C(\alpha)e^{\mathcal{E_Q}(X)}$, where $\mathcal{E_Q}(X)$ is the quantile-based differential entropy and $C(\alpha)=e^{\int_{0}^{1}\log[(\alpha!)^{\frac{1}{\alpha}}p(-\log(Q_3(p)))^\frac{1}{\alpha}] dp}$ is a function of $\alpha$.
    \end{enumerate}
	\begin{proof}
  First part readily follows from the well-known inequality $\log x \leq 1-x$, $\forall x \leq 1$. For obtaining the second part, we employ the log-sum inequality,\\
		\begin{eqnarray}\label{lb}
			\int_{0}^{1} f(Q(p))\log\frac{f(Q(p))}{(\alpha!)^\frac{1}{\alpha}p(-\log(Q_3(p)))^\frac{1}{\alpha}}dQ(p) & \geq & \log \frac{1}{(\alpha!)^\frac{1}{\alpha}\int_{0}^{1}p (-\log(Q_3(p)))^\frac{1}{\alpha} q_X(p)dp} \nonumber\\ 
			& \approx & -\log \mathcal{{CI}_{\alpha}^{Q}}(X,Y). 
		\end{eqnarray}
		Further, the LHS of the above inequality can be expressed as;
		\begin{equation}\label{lb1}
			\int_{0}^{1} f(Q(p))\log\frac{f(Q(p))}{(\alpha!)^\frac{1}{\alpha}p(-\log(Q_3(p)))^\frac{1}{\alpha}}dQ(p)  = -\mathcal{E_Q}(X) - \int_{0}^{1} \log[(\alpha!)^\frac{1}{\alpha}p(-\log(Q_3(p)))^\frac{1}{\alpha}] dp.
		\end{equation}
		From \eqref{lb} and \eqref{lb1},
		\begin{equation}\label{6}
			\log (\mathcal{CI_{\alpha}^{Q}}(X)) \geq  \mathcal{E_Q}(X) + \int_{0}^{1} \log[(\alpha!)^\frac{1}{\alpha}p(-\log(Q_3(p)))^\frac{1}{\alpha}] dp.   
		\end{equation}
		Exponentiating both sides of \eqref{6} we get,
		\begin{equation*}
			\mathcal{CI_{\alpha}^{Q}}(X) \geq C(\alpha)e^{\mathcal{E}_Q(X)}.  
		\end{equation*}
	\end{proof}
    \end{thm}

To find the best statistical models for the data, the usual practice is to apply a transformation to the data. A similar approach to this is to keep the actual data fixed and to transform the QF. In the following result, we provided a tool for computing the FCPI when transformation is applied to the QFs of $X$ and $Y$.
\begin{thm}\label{thm2.1}
Consider $\tau_1(\cdot)$ and $\tau_2(\cdot)$ to be two continuous non-decreasing and invertible transformations, then
\begin{equation}\label{tqfcpi}
\mathcal{CI_\alpha^Q}({\tau_1(X),\tau_2(Y))} \approx (\alpha !)^{\frac{1}{\alpha}} \int_{0}^{1} p(- \log(Q_Y^{-1}(\tau_2^{-1}(\tau_1(Q_X(p))))))^{\frac{1}{\alpha}} d \tau_1(Q_X(p)).
\end{equation}
\begin{proof}
Suppose the CDFs corresponding to $\tau_1(X)$ and $\tau_2(Y)$ be $F_{\tau_1(X)}(x)$ and $F_{\tau_2(Y)}(x)$ respectively. Then, by using \eqref{aqfcpi} we obtain
\begin{equation}\label{tqfcpi1}
\mathcal{CI_\alpha}({\tau_1(X),\tau_2(Y))} \approx (\alpha !)^{\frac{1}{\alpha}}
\int_{0}^{\infty} F_{\tau_{1}(X)}(x)(-\log F_{\tau_{2}(Y)}(x))^{\frac{1}{\alpha}} dx.
\end{equation}
In analogy with \cite{sunoj2018quantile}, the QFs of $F_{\tau_1(X)}(x)$ and $F_{\tau_2(Y)}(x)$ can be shown to be $\tau_1(Q_X(u))$ and $Q_Y^{-1}(\tau_2^{-1}(\tau_1(Q_X(u))))$ respectively. On substituting these in \eqref{tqfcpi1}, the desired result follows. 
\end{proof}
\end{thm}
For the illustration of Theorem 2.1, we provide the following example:
\begin{ex}
Let $X$ and $Y$ be two independent reciprocal exponential rvs with respective QFs $Q_X(u)=-\lambda_1/\ln u$ and $Q_Y(u)=-\lambda_2/\ln u$, where $\lambda_1, \lambda_2 >0$ and $0<u<1$. Consider the transformation $\tau_1(X)=e^{-\lambda_1/X}$ and $\tau_2(Y)=e^{-\lambda_2/Y}$. Then, $Q_Y^{-1}(\tau_2^{-1}(\tau_1(Q_X(u))))=u$. Thus, from \eqref{tqfcpi}, we obtain
\begin{eqnarray*}
\mathcal{CI_\alpha^Q}({e^{-\lambda_1/X}, e^{-\lambda_2/Y}}) &\approx &(\alpha !)^{\frac{1}{\alpha}}\int_{0}^{1} p(-\log p)^{\frac{1}{\alpha}} dp\\
{}& = & \frac{\Gamma(\alpha+1)^\frac{1}{\alpha}\left(\frac{1}{\alpha}\right)!}{2^{1+1/\alpha}} 
\end{eqnarray*}
\end{ex}

It is quite relevant to establish a relation between our proposed measure and other well-known measures. In this context, we introduce quantile-based fractional cumulative Kullback-Leibler divergence (KL-divergence) and quantile-based fractional cumulative past entropy of order $\frac{1}{\alpha}$.\\

\cite{saha2025fractional} proposed fractional cumulative past KL-divergence of order $\frac{1}{\alpha}$ corresponding to rvs $X$ and $Y$ with CDF $F(x)$ and $G(x)$ as;
\begin{equation}\label{FKL}
\mathcal{KL_\alpha}(X,Y)=\int_{0}^{\infty} F(x)\left[Ln_\alpha \frac{F(x)}{G(x)}\right]^\frac{1}{\alpha}dx+E(X)-E(Y)
\end{equation}
We define the corresponding quantile-based version of \eqref{FKL} as;
\begin{equation}\label{qKL}
\mathcal{KL_\alpha^Q}(X,Y)=\int_{0}^{1} p\left[Ln_\alpha \frac{p}{Q_3(p)}\right]^\frac{1}{\alpha} q_X(p)dp+E(X)-E(Y)
\end{equation}
In a similar line, the quantile version of \eqref{EFCPE}, denoted as Q-EFCPE for $X$ is given by
\begin{equation}\label{qfcpe}
\mathcal{\overline{CE}_\alpha^Q}(X)=\int_{0}^{1}p[-Ln_\alpha p]^\frac{1}{\alpha} q_X(p)dp.
\end{equation}
The following proposition indicates a connection between the measure we introduced and some well-known measures like Q-EFCPE and quantile-based fractional KL-divergence measure.
\begin{prop}
The $\mathcal{CI_\alpha^Q}(X,Y)$ can be expressed as;
\begin{equation*}
\mathcal{CI_\alpha^Q}(X,Y)=\mathcal{\overline{CE}_\alpha^Q}(X)+(-1)^{\frac{1}{\alpha}+1}\{\mathcal{KL_\alpha^Q}(X,Y)+E(X)-E(Y)\}, \hspace{0.2cm} 0<\alpha\leq1,\\
\end{equation*}
where $\mathcal{\overline{CE}_\alpha^Q}(X)$ and $\mathcal{KL_\alpha^Q}(X,Y)$ are given in \eqref{qKL} and \eqref{qfcpe}, respectively.
\begin{proof}
From \eqref{qfcpe}, we have
\begin{eqnarray}\label{prop 2.1}
\mathcal{\overline{CE}_\alpha^Q}(X)&=&\int_{0}^{1}p[-Ln_\alpha p]^\frac{1}{\alpha} q_X(p)dp\\\nonumber
{}&=& \int_{0}^{1} p\bigg[-Ln_\alpha \frac{p}{Q_3(p)} Q_3(p)\bigg]^\frac{1}{\alpha} q_X(p)dp.\nonumber
\end{eqnarray}
By using the result $[Ln_\alpha(mn)]^\frac{1}{\alpha}=[Ln_\alpha(m)]^\frac{1}{\alpha}+[Ln_\alpha(n)]^\frac{1}{\alpha}$, \eqref{prop 2.1} can be rewritten as
\begin{eqnarray*}
 \mathcal{\overline{CE}_\alpha^Q}(X)&=& \int_{0}^{1} p[-Ln_\alpha Q_3(p)]^\frac{1}{\alpha} q_X(p)dp + \int_{0}^{1} p\bigg[-Ln_\alpha\frac{p}{Q_3(p)}\bigg]^\frac{1}{\alpha} q_X(p)dp \\
 {}&=& \mathcal{CI_\alpha^Q}(X,Y)+(-1)^\frac{1}{\alpha} \{\mathcal{KL_\alpha^Q}(X,Y)-E(Y)+E(X)\}. 
\end{eqnarray*}
Thus, 
\begin{equation*}
\mathcal{CI_\alpha^Q}(X,Y)=\mathcal{\overline{CE}_\alpha^Q}(X)+(-1)^{\frac{1}{\alpha}+1}\{\mathcal{KL_\alpha^Q}(X,Y)+E(X)-E(Y)\}.
\end{equation*}
Hence, the proof.
\end{proof}
\end{prop}
When $\alpha=1$, the relation in Proposition 2.1 reduces to the following expression:
\begin{equation}
\mathcal{CI^Q}(X,Y)=\mathcal{\overline{CE}^Q}(X)+\mathcal{KL^Q}(X,Y)+E(X)-E(Y),
\end{equation}
where $\mathcal{CI^Q}(X,Y)$, $\mathcal{\overline{CE}^Q}(X,Y)$ and $\mathcal{KL^Q}(X,Y)$ are cumulative past inaccuracy, cumulative past entropy and cumulative past Kullback-Leibler divergence based on quantile functions, respectively. \\ 

Affine transformations play a vital role in statistics and probability due to their ability to condense, standardize, and preserve the structure of data. They enable effective analysis, explanation, and visualization of datasets that are complex in nature by employing many fundamental statistical techniques and models. The affine transformation corresponding to two rvs $X$ and $Y$ is defined as $Y=aX+b$, $a>0$ and $b \geq 0$. Now, we illustrate the effectiveness of the Q-FCPI in \eqref{qfcpi} under an affine transformation.
\begin{prop}
Suppose $X^{\prime}=aX+b$ and  $Y^{\prime}=aY+b$, $a>0$ and $b \geq 0$  are affine transformations of rvs $X$ and $Y$ respectively. Then, for $0 <\alpha \leq 1$, we obtain
\begin{equation*}
 \mathcal{CI_\alpha^Q}(X^\prime,Y^\prime)=a \mathcal{CI_\alpha^Q}(X,Y).   
\end{equation*}
\end{prop}
\begin{proof}
Let $F_{X^\prime}(x)=F_{X}(\frac{x-b}{a})$ and $F_{Y^\prime}(x)=F_Y(\frac{x-b}{a})$ be the CDFs corresponding to $X^\prime$ and $Y^\prime$ respectively. Now, from the definition of FCPI in (), we obtain
\begin{eqnarray} \label{affn trans}
\mathcal{CI_\alpha}(X^\prime, Y^\prime)&=&\int_{b}^{\infty} F_{X^\prime}(u)[-Ln_{\alpha} F_{Y^\prime}(u)]^\frac{1}{\alpha} du\\ \nonumber
{}&=& \int_{b}^{\infty} F_{X}\left(\frac{u-b}{a}\right)\left[-Ln_{\alpha}F_{Y}\left({\frac{u-b}{a}}\right)\right]^\frac{1}{\alpha} du\\ \nonumber
{}&=&a \int_{0}^{\infty} F_X(u)[-Ln_\alpha F_Y(u)]^\frac{1}{\alpha} du\\
{}&=& a \mathcal{CI_\alpha}(X,Y)\nonumber
\end{eqnarray}
Then, the corresponding quantile version of \eqref{affn trans} is;
\begin{eqnarray*}
\mathcal{CI_\alpha^Q}(X^\prime,Y^\prime)&=&a \int_{0}^{1}p [-Ln_\alpha Q_3(p)]^\frac{1}{\alpha} q_X(p)dp\\
{}&=& a \mathcal{CI_\alpha^Q}(X,Y).
\end{eqnarray*}
Thus, the proof is completed.
\end{proof}
\begin{lem}\label{Lem 2.1}
    Let $\alpha$ be a rational number in (0,1), such that $\frac{1}{\alpha}$ is an integer. Then, the function
\begin{equation}\label{lem con}
\mathcal{W}_{\alpha,Y}(u)=\int_{u}^{1}[-Ln_\alpha Q_3(p)]^\frac{1}{\alpha} q_X(p) dp
\end{equation}
is\\
(i) decreasing and convex, if $\frac{1}{\alpha}$ is odd;\\
(ii) decreasing and concave in u, if $\frac{1}{\alpha}$ is even.
\begin{proof}
 Differentiating $\mathcal{W}_{\alpha,Y}(u)$ in \eqref{lem con} with respect to $u$, we obtain
 \begin{equation}\label{dif}
 \mathcal{W}^{\prime}_{\alpha,Y}(u)=\frac{d}{du} \mathcal{W}_{\alpha,Y}(u) = -(-Ln_\alpha Q_3(u))^\frac{1}{\alpha} \leq 0,
 \end{equation}
implying with respect to $0 < u < 1$, $\mathcal{W}_{\alpha,Y}(u)$ is decreasing. Since, $Ln_\alpha x$ is the inverse of MLF, say $g^{-1}(x)$, i.e, $Ln_\alpha x=g^{-1}(x)\implies Ln_{\alpha} Q_3(u) =g^{-1}(Q_3(u))$. Hence, from \eqref{dif}, for $\alpha \in (0,1)$, we obtain
\begin{eqnarray}\label{dif 2}
 \mathcal{W}^{\prime\prime}_{\alpha,Y}(u)&=& \frac{d}{du}\left\{-(-g^{-1}(Q_3(u)))^\frac{1}{\alpha}\right\}\\\nonumber
 {}&=& (-1)^{1+\frac{1}{\alpha}} \frac{q_3(u)}{g^{\prime}[g^{-1}(Q_3(u))]}\left\{
\begin{array}{ll}
\geq 0, \hspace{0.15cm}if \hspace{0.15cm} \frac{1}{\alpha} \hspace{0.15cm} \mbox{is odd} ;\\
\leq 0, \hspace{0.15cm}if \hspace{0.15cm} \frac{1}{\alpha} \hspace{0.15cm}  \mbox{is even}.
\end{array}
\right.
\end{eqnarray}
Hence, from \eqref{dif 2}, it is understood that $\mathcal{W}_{\alpha,Y}(u)$ is convex in $u$, if $\frac{1}{\alpha}$ is an odd number and concave in $u$, if $\frac{1}{\alpha}$ is an even number. Thus, the proof is completed.
\end{proof}
\end{lem}
\begin{prop}\label{prop 2.3}
Suppose $X$ and $Y$ have QFs $Q_X(\cdot)$ and $Q_Y(\cdot)$, respectively. Then,
\begin{eqnarray}
\mathcal{CI_\alpha^Q}(X,Y)=E_X(\mathcal{W}_{\alpha,Y}(u)).
\end{eqnarray}
\begin{proof}
From the definition of Q-FCPI in \eqref{qfcpi}, we have
\begin{eqnarray*}
\mathcal{CI_\alpha^Q}(X,Y)&=&\int_{0}^{1}p[-Ln_{\alpha}(Q_3(p))]^\frac{1}{\alpha} q_X(p)dp\nonumber\\
{}&=& \int_{0}^{1}\left[\int_{0}^{p}du\right][-Ln_{\alpha}(Q_3(p))]^\frac{1}{\alpha} q_X(p)dp\nonumber\\
{}&=& \int_{0}^{1}\left[\int_{u}^{1}[-Ln_{\alpha}(Q_3(p))]^\frac{1}{\alpha} q_X(p)dp\right]du\nonumber\\
{}&=& E_X(\mathcal{W}_{\alpha,Y}(u)),
\end{eqnarray*}
which gives the desired result.
\end{proof}
\end{prop}
\subsection{Stochastic Orders}
In probability and statistics, stochastic orders are crucial tools for comparing rvs based on variety of statistical measures useful in different domains of applications. They pave the way to more robust and comprehensible assessments that improve decision-making and risk management. Suppose $Q_X(\cdot)$ and $Q_Y(\cdot)$ are QFs corresponding to rvs $X$ and $Y$ with respective CDFs $F_X(\cdot)$ and $F_Y(\cdot)$. Then, $X$ is smaller than $Y$
\begin{itemize}
        \item [(i)] in stochastic ordering, represented as $X \leq_{st} Y$ if and only if $F(x) \geq G(x)$ or equivalently $Q_X(u) \leq Q_Y(u)$ $\forall$ $u$ in (0,1);
    \item [(ii)] in decreasing convex ordering, represented as $X \leq_{dcx} Y$ if $E(\psi(X)) \leq E(\psi(Y))$, for any convex decreasing function $\psi(\cdot)$ such that the mean exist;
\end{itemize}
\begin{prop}
Suppose $X$ and $Y$ have QFs $Q_X(\cdot)$ and $Q_Y(\cdot)$, respectively. 
\begin{itemize}
    \item [(A)] If $X \leq_{st} Y$, we get $\mathcal{CI_\alpha^Q}(X,Y) \geq max\{\mathcal{\overline{CE}_\alpha^Q}(X), \mathcal{\overline{CE}_\alpha^Q}(Y)\}$;
    \item [(B)] If $X \geq_{st} Y$, we get  $\mathcal{CI_\alpha^Q}(X,Y) \leq min\{\mathcal{\overline{CE}_\alpha^Q}(X), \mathcal{\overline{CE}_\alpha^Q}(Y)\}$;
\end{itemize}
where $\mathcal{\overline{CE}_\alpha^Q}(X)$ and $\mathcal{\overline{CE}_\alpha^Q}(Y)$ are the Q-EFCPEs of $X$ and $Y$, respectively.
\begin{proof}
(A) Under the assumptions made, $X \leq_{st} Y\implies F(x) \geq G(x) \implies u \geq Q_3(u)$ $\forall \hspace{0.2cm} u$ in (0,1). The inverse MLF $Ln_{\alpha}(x)$ is increasing function with respect to $x$ for $0 < \alpha < 1$ (see \cite{liang2022computation}). Thus,
\begin{eqnarray}\label{st 1}
Ln_{\alpha} p &\geq& Ln_{\alpha} Q_3(p)\\\nonumber
\implies \int_{0}^{1} p [-Ln_{\alpha}p]^\frac{1}{\alpha}q_X(p)dp &\leq &\int_{0}^{1} p [-Ln_{\alpha}Q_3(p)]^\frac{1}{\alpha}q_X(p)dp  \\ \nonumber  
\implies \mathcal{\overline{CE}_\alpha^Q}(X) &\leq &\mathcal{CI_\alpha^Q}(X,Y).
\end{eqnarray}
Moreover, $p \geq Q_3(p)$ and $[-Ln_{\alpha}Q_3(p)]^\frac{1}{\alpha} \geq 0$ for $0<\alpha<1$. Thus, we have
\begin{eqnarray}\label{st 2}
p[Ln_{\alpha}Q_3(p)]^\frac{1}{\alpha}& \geq & Q_3(p)[Ln_{\alpha}Q_3(p)]^\frac{1}{\alpha}\\\nonumber
\implies \int_{0}^{1} p [-Ln_{\alpha}Q_3(p)]^\frac{1}{\alpha}q_X(p)dp &\geq &\int_{0}^{1} Q_3(p) [-Ln_{\alpha}Q_3(p)]^\frac{1}{\alpha}q_X(p)dp  \\ \nonumber  
\implies \mathcal{CI_\alpha^Q}(X,Y) &\geq &\mathcal{\overline{CE}_\alpha^Q}(Y).
\end{eqnarray}
On combining \eqref{st 1} and \eqref{st 2}, the required result follows.

(B) The proof of this part is the same as that of Part (A). Therefore, it is omitted.
\end{proof}
\end{prop}
\begin{prop}
Consider $X$ and $Y$ with QFs $Q_X(\cdot)$ and $Q_Y(\cdot)$ respectively. If $X \leq_{dcx} (\leq_{dcv})Y$, then for $0 <\alpha <1$,\\
\begin{equation}
 \mathcal{CI_\alpha^Q}(X,Y) \leq \mathcal{\overline{CE}_\alpha^Q}(Y),  \mbox{when $\frac{1}{\alpha}$ is odd (even).}   
\end{equation}
\begin{proof}
From, Lemma~\ref{Lem 2.1}, the function $\mathcal{W}_{\alpha,Y}(u)$ is decreasing convex (concave) with respect to $u$ when $\frac{1}{\alpha}$ is odd (even). Now, under our assumption, we obtain
$$E_X(\mathcal{W}_{\alpha,Y}(u)) \leq E_Y(\mathcal{W}_{\alpha,Y}(u))=\mathcal{CI_\alpha^Q}(Y,Y)=\mathcal{\overline{CE}_\alpha^Q}(Y) $$
Thus, the remaining part of the proof can be easily obtained from Proposition~\ref{prop 2.3}.
\end{proof}
\end{prop}
\section{Quantile-based dynamic FCPI measure}
In the previous section, we discussed the Q-FCPI measure and studied its various properties. There are some situations in reliability and survival analysis where the current ages of two systems are taken into account. In such situations, the measure defined in \eqref{qfcpi} may not be appropriate, but dynamic measures are useful to explain the uncertainty of random lifetimes when age changes. Let $X$ be the random lifetime of a system. Then $[t-X|X \leq t], \hspace{0.1cm} t>0$ is the system's inactivity time. 
The inactivity time of a system is the duration of time between the inspection time $t$ and the failure time $X$, given that at time $t$ the system is down (see \cite{saha2025fractional}). Here, rv $[t-X|X \leq t]$ has relevance in economics, as it reflects the income distribution of the poor for a poverty line $t$. Motivated by these, \cite{saha2025fractional} proposed a dynamic fractional cumulative past inaccuracy (DFCPI) measure based on the inverse MLF function defined as
\begin{equation}\label{DFCPI}
\mathcal{CI_{\alpha}}(X,Y;t)=\int_{0}^{t} \frac{F(x)}{F(t)}\left[-Ln_{\alpha} \frac{G(x)}{G(t)}\right]^\frac{1}{\alpha}dx, \hspace{0.1cm}0<\alpha<1, \hspace{0.1cm}t>0.
\end{equation}
\begin{defn}
 Let $X$ and $Y$ be two nonnegative and absolutely continuous rvs with QFs $Q_1(\cdot)$ and $Q_2(\cdot)$, respectively. Then, the quantile-based dynamic fractional cumulative past inaccuracy (Q-DFCPI) function corresponding to \eqref{DFCPI} is,  
 \begin{eqnarray}\label{Q-DFCPI}
\mathcal{CI_\alpha^Q}(X,Y;u)&=&\int_{0}^{u} \frac{F(Q_X(p))}{F(Q_X(u))}\left[-Ln_{\alpha} \frac{G(Q_X(p))}{G(Q_X(u))}\right]^\frac{1}{\alpha}dQ_X(p)\\
{}&=& \frac{1}{u}\int_{0}^{u} p\left[-Ln_{\alpha} \frac{Q_3(p)}{Q_3(u)}\right]^\frac{1}{\alpha}q_X(p)dp, \hspace{0.1cm}0<\alpha<1, \hspace{0.1cm}0<u<1
 \end{eqnarray}
\end{defn}
When $u \rightarrow 1$, the Q-DFCPI in \eqref{Q-DFCPI} tends to the Q-FCPI measure in \eqref{qfcpi}. The Q-DFCPI measure is approximated as follows using the approximation $Ln_{\alpha}x \approx log x^{\alpha!}$ $(0<\alpha<1)$, 
\begin{equation}
\mathcal{CI_\alpha^Q}(X,Y;u)\approx\frac{(\alpha!)^\frac{1}{\alpha}}{u}\int_{0}^{u} p\left[-\log \frac{Q_3(p)}{Q_3(u)}\right]^\frac{1}{\alpha}q_X(p)dp. 
\end{equation}
Next, in line with \eqref{qmfcpi}, we propose another inaccuracy measure, say quantile-based modified dynamic cumulative past inaccuracy (Q-MDFCPI), given by
\begin{equation}\label{Q-MDFCPI}
\mathcal{\widetilde{CI}_\alpha^Q}(X,Y;u)=\frac{1}{u}\int_{0}^{u} p\left[-Ln_{\alpha} \frac{Q_3(p)}{Q_3(u)}\right]q_X(p)dp.
\end{equation}
\begin{rem}
Assuming $X$ and $Y$ are rvs with respective QFs $Q_X(\cdot)$ and $Q_Y(\cdot)$. For $0<\alpha<1$ and $0\leq q_X(u)\leq 1$, the Q-DFCPI and Q-MDFCPI satisfy the relationship
$$\mathcal{CI_\alpha^{Q}}(X,Y;u) \geq \left[\mathcal{\widetilde{CI}_{\alpha}^Q}(X,Y;u)\right]^\frac{1}{\alpha} $$
\end{rem}
\begin{ex}
Suppose $X$ and $Y$ satisfy PRHM as given in \eqref{PRHM}. Then, corresponding to $X$ and $Y$ the Q-DFCPI is given by
\begin{eqnarray}\label{Q-DFCPI}
\mathcal{CI_\alpha^Q}(X,Y;u)&=& \frac{1}{u}\int_{0}^{u} p\left[-Ln_{\alpha} \frac{p^\theta}{u^\theta}\right]^\frac{1}{\alpha}q_X(p)dp\\ \nonumber
{}& = & \frac{\theta}{u}\int_{0}^{u}p \left[-Ln_{\alpha} \frac{p}{u}\right]^\frac{1}{\alpha}q_X(p)dp= \theta\mathcal{\widetilde{CE}_\alpha^Q}(X;u)\nonumber
\end{eqnarray}
where $\mathcal{\widetilde{CE}_\alpha^Q}(X;u)$ is the quantile version of dynamic extended fractional cumulative past entropy. 
\end{ex} 
Next, we provide a characterization theorem that provides bounds for $\mathcal{CI_\alpha^Q}(X,Y;u)$ in terms of reversed hazard quantile function and reversed mean residual quantile function.
\begin{prop}
For the rvs $X$ and $Y$ with respective QFs $Q_X(\cdot)$ and $Q_Y(\cdot)$, the Q-DFCPI measure for $0<u<1$ and $0<\alpha<1$ is increasing (decreasing) in $u$ if and only if 
\begin{equation*}
\mathcal{CI_\alpha^Q}(X,Y;u) \leq (\geq)\mathcal{M}_{XQ}(u)[-Ln_\alpha Q_3(u)]^{\frac{1}{\alpha}-1}\times \frac{\alpha!}{\alpha}\times \frac{1}{\tilde{h}_3(u)Q_3(u)},
\end{equation*}
where $\mathcal{M}_{XQ}(u)$ is given in  \eqref{Q-MPL} and $\tilde{h}_3(u)=[uq_3(u)]^{-1}$
\end{prop}
\begin{proof}
Using $Ln_{\alpha}(m^b)=b^\alpha Ln_\alpha (m)$ and $[Ln_\alpha(mn)]^\frac{1}{\alpha}=[Ln_{\alpha}(m)]^\frac{1}{\alpha}+[Ln_{\alpha}(n)]^\frac{1}{\alpha}$ in \eqref{Q-DFCPI} we obtain
\begin{equation}\label{LB DFCPI}
\mathcal{CI_\alpha^Q}(X,Y;u)=\int_{0}^{u}\frac{p}{u}[-Ln_{\alpha}Q_3(p)]^\frac{1}{\alpha}q_X (p)dp-\int_{0}^{u}\frac{p}{u}[-Ln_{\alpha}Q_3(u)]^\frac{1}{\alpha}q_X (p)dp
\end{equation}
Differentiating both sides of \eqref{LB DFCPI} with respect to u using the Leibniz rule, we get
\begin{equation}\label{LB DFCPI2}
\frac{d}{du}\mathcal{CI_\alpha^Q}(X,Y;u)=\frac{\alpha! \hspace{0.1cm}\mathcal{M}_{XQ}(u)q_3(u)[-Ln_{\alpha}Q_3(u)]^{\frac{1}{\alpha}-1}}{\alpha \hspace{0.1cm} Q_3(u)}-\frac{\mathcal{CI_\alpha^Q}(X,Y;u)}{u}
\end{equation}
Assume that $\mathcal{CI_\alpha^Q},(X,Y;u)$ is increasing, then $\frac{d}{du}\mathcal{CI_\alpha^Q}(X,Y;u) \geq 0$ and also after substituting $\tilde{h}_3(u)=[uq_3(u)]^{-1}$ into \eqref{LB DFCPI2}, the result follows easily. Thus, the proof of one part is obtained. The proof of the other part readily follows.\end{proof}
\begin{prop}
For rvs $X$ and $Y$, we obtain 
\begin{equation*}
\mathcal{CI_\alpha^Q},(X,Y;u)=E_X(\mathcal{W}_{\alpha,Y}(v;u)|v\leq u), \hspace{0.1cm}0<u<1\hspace{0.1cm} \mbox{and}\hspace{0.1cm} 0<\alpha<1,
\end{equation*}
where $\mathcal{W}_{\alpha,Y}(v;u)=\int_{v}^{u} [-Ln_{\alpha}\frac{Q_3(p)}{Q_3(u)}]^\frac{1}{\alpha}q_X(p)dp$.
\end{prop}
\begin{proof}
From \eqref{Q-DFCPI}, we obtain
 \begin{eqnarray}\label{Q-DFCPI}
\mathcal{CI_\alpha^Q}(X,Y;u)&=&\int_{0}^{u} \frac{p}{u}\left[-Ln_{\alpha} \frac{Q_3(p)}{Q_3(u)}\right]^\frac{1}{\alpha}q_X(p)dp\\\nonumber
{}&=& \int_{0}^{u}\left(\int_{0}^{p} dv\right)\frac{1}{u}\left[-Ln_{\alpha} \frac{Q_3(p)}{Q_3(u)}\right]^\frac{1}{\alpha}q_X(p)dp\\\nonumber
{}&=& \int_{0}^{u}\frac{1}{u}\left(\int_{v}^{u}\left[-Ln_{\alpha} \frac{Q_3(p)}{Q_3(u)}\right]^\frac{1}{\alpha}q_X(p)dp\right)dv \hspace{0.2cm} (\mbox{using Fubini's theorem})\\\nonumber
{}&=& E_X(\mathcal{W}_{\alpha,Y}(v;u)|v\leq u).
\end{eqnarray}
Hence, the result.
\end{proof}

In the following theorem, we can elucidate the computation of Q-DFCPI when a transformation is introduced to the QFs of $X$ and $Y$.
\begin{thm}\label{thm2.1}
Consider $\eta_1(\cdot)$ and $\eta_2(\cdot)$ to be two continuous non-decreasing and invertible transformations, then
\begin{equation}\label{tqfcpi}
\mathcal{CI_\alpha^Q}({\eta_1(X),\eta_2(Y))} =  \int_{0}^{u} \frac{p}{u}\left[- Ln_\alpha\left(\frac{Q_Y^{-1}(\eta_2^{-1}(\eta_1(Q_X(p)))))}{Q_Y^{-1}(\eta_2^{-1}(\eta_1(Q_X(u)))))}\right)\right]^{\frac{1}{\alpha}} d \eta_1(Q_X(p)).
\end{equation}
\begin{proof}
The proof of the theorem is similar to that of Theorem 2.2. Hence, it is not provided. 
\end{proof}
\end{thm}
\begin{prop}
For the rvs $X$ and $Y$, we get
\begin{itemize}
    \item [(i)] If $X \leq_{st} Y$, then $\mathcal{CI_\alpha^Q}(X,Y;u) \geq \mathcal{\widetilde{CE}_\alpha^Q}(X;u)+\mathcal{M}_{XQ}(u)\left[-Ln_\alpha \frac{u}{Q_3(u)}\right]^\frac{1}{\alpha}, \hspace{0.1cm}0<\alpha<1$;
    \item [(ii)] If $X \geq_{st} Y$, then $\mathcal{CI_\alpha^Q}(X,Y;u) \leq \mathcal{\widetilde{CE}_\alpha^Q}(X;u)+\mathcal{M}_{XQ}(u)\left[-Ln_\alpha \frac{u}{Q_3(u)}\right]^\frac{1}{\alpha}, \hspace{0.1cm}0<\alpha<1$,
\end{itemize}
where $\mathcal{\widetilde{CE}_\alpha^Q}(X;u)=\int_{0}^{u}\frac{p}{u}\left[-Ln_\alpha\frac{p}{u}\right]^\frac{1}{\alpha}q_X(p)dp$ is the quantile version of dynamic extended fractional cumulative past entropy (see[\cite{saha2023extended}]).
\end{prop}
\begin{proof}
(i) Under the assumption, $X \leq_{st} Y \implies F(x) \geq G(x)\implies u\geq Q_3(u) \implies [-Ln_\alpha u]^\frac{1}{\alpha} \leq [-Ln_\alpha Q_3(u)]^\frac{1}{\alpha}$ for $0<\alpha<1$. Now, from \eqref{Q-DFCPI}, 
\begin{eqnarray*}
\mathcal{CI_\alpha^Q}(X,Y;u)&=&\int_{0}^{u} \frac{p}{u}([-Ln_{\alpha} Q_3(p)]^\frac{1}{\alpha}+[-Ln_{\alpha} (1/{Q_3(u)})]^\frac{1}{\alpha})q_X(p)dp\nonumber
\end{eqnarray*} 
\hspace{5cm}(Since $[Ln_\alpha(mn)]^\frac{1}{\alpha}=[Ln_{\alpha}(m)]^\frac{1}{\alpha}+[Ln_{\alpha}(n)]^\frac{1}{\alpha}$)
\begin{eqnarray}\label{st 2}
\hspace{4cm}{}&=& \int_{0}^{u} \frac{p}{u}[-Ln_\alpha Q_3(p)]^\frac{1}{\alpha}q_X(p) dp +\int_{0}^{u} \frac{p}{u}[-Ln_\alpha (1/{Q_3(u))}]^\frac{1}{\alpha}q_X(p) dp\nonumber\\
{}& \geq & \int_{0}^{u} \frac{p}{u}[-Ln_{\alpha} p]^\frac{1}{\alpha}q_X(p) dp + \mathcal{M}_{XQ}(u)[-Ln_\alpha (1/{Q_3(u)})]^\frac{1}{\alpha}\nonumber\\
{}&=& \mathcal{\widetilde{CE}_\alpha^Q}(X;u) + \mathcal{M}_{XQ}(u) \left\{[-Ln_\alpha u]^\frac{1}{\alpha}+[-Ln_\alpha (1/Q_3(u))]^\frac{1}{\alpha}\right\}
\end{eqnarray}
The result can be easily deduced from \eqref{st 2}. Thus, the proof.

The proof of Part (ii) is the same as that of Part (i), and therefore it is omitted. 
\end{proof}
\begin{prop}
Suppose $Q_X(\cdot)$ and $Q_Y(\cdot)$ be the QFs of $X$ and $Y$ and $Q_3(u)=Q_{Y}^{-1}(Q_X(u))$ respectively. For $0<u<1$ and $0<\alpha<1$,
\begin{itemize}
\item [(i)] Let $X \leq_{rh} Y$, then $\mathcal{CI_\alpha^Q}(X,Y;u) \geq max \left\{\mathcal{\widetilde{CE}_\alpha^Q}(X;u);\mathcal{\widetilde{CE}_\alpha^Q}(Y;u)\right\}$;
\item [(ii)] Let $X \geq_{rh} Y$, then $\mathcal{CI_\alpha^Q}(X,Y;u) \leq min \left\{\mathcal{\widetilde{CE}_\alpha^Q}(X;u);\mathcal{\widetilde{CE}_\alpha^Q}(Y;u)\right\}$,
\end{itemize}
where $\mathcal{\widetilde{CE}_\alpha^Q}(X;u)$ is defined in Proposition 3.3.
\end{prop}
\begin{proof}
(i) Given $X\leq_{rh}Y$ which is equivalent to $\frac{F(Q_X(p))}{F(Q_X(u))}\geq \frac{G(Q_X(p))}{G(Q_X(u))}{}\implies\frac{p}{u} \geq \frac{Q_3(p)}{Q_3(u)}$ for all $ Q_X(p)\leq Q_X(u)$. Now, 
\begin{eqnarray}\label{RH1}
\mathcal{CI_\alpha^Q}(X,Y;u)&=&\int_{0}^{u} \frac{p}{u}\left[-Ln_{\alpha} \frac{Q_3(p)}{Q_3(u)}\right]^\frac{1}{\alpha}q_X(p)dp\nonumber\\
{}&\geq& \int_{0}^{u} \frac{Q_3(p)}{Q_3(u)}\left[-Ln_{\alpha} \frac{Q_3(p)}{Q_3(u)}\right]^\frac{1}{\alpha}q_X(p)dp= \mathcal{\widetilde{CE}_\alpha^Q}(Y;u).
\end{eqnarray}
\hspace{1.8cm}
Since, $\frac{p}{u}\geq \frac{Q_3(p)}{Q_3(u)} \implies [-Ln_\alpha \frac{p}{u}]^\frac{1}{\alpha} \leq [-Ln_\alpha \frac{Q_3(p)}{Q_3(u)}]^\frac{1}{\alpha}$, we obtain
\begin{eqnarray}\label{RH2}
\mathcal{CI_\alpha^Q}(X,Y;u)&=&\int_{0}^{u} \frac{p}{u}\left[-Ln_{\alpha} \frac{Q_3(p)}{Q_3(u)}\right]^\frac{1}{\alpha}q_X(p)dp\nonumber\\
{}&\geq& \int_{0}^{u}  \frac{p}{u}\left[-Ln_{\alpha}  \frac{p}{u}\right]^\frac{1}{\alpha}q_X(p)dp= \mathcal{\widetilde{CE}_\alpha^Q}(X;u).
\end{eqnarray}
The result can be easily deduced by combining \eqref{RH1} and \eqref{RH2}

(ii) The proof of this part can be similarly obtained. Therefore, it is skipped.
\end{proof}
\section{Nonparametric estimation of Q-FCPI}	
Suppose $X_{(1)}, X_{(2)},\ldots, X_{(n)}$ denote the order statistics corresponding to lifetime of $n$ iid components with common CDF $F(x)$ and QF $Q_{X}(u)$. Also, $Y_{(1)}, Y_{(2)},\ldots, Y_{(n)}$ denote the order statistics corresponding to $n$ iid components with common CDF $G(x)$ and QF $Q_{Y}(u)$. Assume the empirical distribution functions of $X$ and $Y$ be $\hat{F}(x_{(i)})$ and $\hat{G}(x_{(i)})$.
\cite{parzen1979nonparametric} introduced an empirical quantile function which is a step function with jump $\frac{1}{n}$ given by
\begin{equation*}
\bar{Q}_{X}(u)=X_{(k)},\hspace{0.2cm}\frac{k-1}{n}<u<\frac{k}{n},\hspace{0.2cm}k=1,2,\ldots,n,
\end{equation*}
 and this estimator's smoothed version is provided by,
\begin{equation*}
\bar{Q}_n(u)=n\bigg(\frac{k}{n}-u\bigg)X_{(k-1)}+n\bigg(u-\frac{k-1}{n}\bigg)X_{(k)}
\end{equation*}
for $\frac{k-1}{n}<u<\frac{k}{n}$, $k=1,2,\ldots,n$. They also defined corresponding estimator of the empirical qdf as
\begin{equation}\label{eqdf}
\bar{q}_n{(u)}=\frac{d}{du}\bar{Q}_n(u)=n(X_{(k)}-X_{(k-1)}), \: \textrm{for} \:\frac{k-1}{n}<u<\frac{k}{n}.
\end{equation}
The empirical plug-in estimator of $Q_3(u)=Q_{Y}^{-1}(Q_{X}(u))$ is
\begin{equation}\label{hat Q_3}
\hat{Q}_3(u)= \hat{G}(\bar{Q}_{n}(u))
\end{equation}
Using \eqref{eqdf} and \eqref{hat Q_3}, the nonparametric estimator of Q-FCPI becomes
\begin{equation}\label{est qfcri}
\hat{\mathcal{{CI}_\alpha^Q}}(X,Y) = \int_{0}^{1}p[-Ln_{\alpha}\hat{Q}_3(p)]^\frac{1}{\alpha}  \bar{q}_n(p)dp  
\end{equation}
Now, by approximating the integral to summation in \eqref{est qfcri}, the plug-in estimator of Q-FCPI is
\begin{equation}\label{est fcrir2}
\hat{\mathcal{{CI}_\alpha^Q}}(X,Y)=\displaystyle\sum\limits_{i=1}^{n}\hat{F}(X_{(i)})[-Ln_{\alpha}(\hat{Q}_3(u))]^\frac{1}{\alpha} n(X_{(i)}-X_{(i-1)})(S_{(i)}-S_{(i-1)})
\end{equation}
where $\hat{F}(X_{(i)})$ is the empirical distribution function of rv $X$ and $S_{(i)}$ is defined by
\[S_{(i)}
=
\begin{cases}
0, i=0  \\
\hat{F}(X_{(i)})=\frac{i}{n}, i=1,2,3,\ldots,n-1 \\
1,i=n
\end{cases}
\]
Hence, the simplified expression of \eqref{est fcrir2} is
\begin{eqnarray}\label{EQFCPI}
\hat{\mathcal{{CI}_\alpha^Q}}(X,Y)&
=& 
\displaystyle\sum\limits_{i=1}^{n-1}
\left(\frac{i}{n}\right)
\left[-Ln_\alpha\left(\hat{G}(X_{(i)})\right)\right]^\frac{1}{\alpha}
\, \,(X_{(i)}-X_{(i-1)})\,\\
{}&\approx&
(\alpha!)^\frac{1}{\alpha}\displaystyle\sum\limits_{i=1}^{n-1}
\left(\frac{i}{n}\right)
\left[-\log\left(\hat{G}(X_{(i)})\right)\right]^\frac{1}{\alpha}
\, \,(X_{(i)}-X_{(i-1)})\,\nonumber
\end{eqnarray}
\section{Simulation study}
 In this section, we carry out simulations to assess the performance of the non-parametric estimator $\hat{\mathcal{{CI}_\alpha^Q}}(X,Y)$ in \eqref{EQFCPI}. Initially, we generated pairs of random samples for different sample sizes considering the true model $F$ as power-Pareto distribution having QF $Q_X(u)=Cu^{\lambda_1}(1-u)^{-\lambda_2}, C,\lambda_1,\lambda_2 >0$ which is a quantile model and the assigned distribution $G$ as an exponential distribution with parameter $\theta$. Then, we calculate the estimated value of Q-FCPI along with its bias and MSE for various values of fractional parameter $\alpha$ (see Table ~\ref{Sim 1}).
 An increase in the sample size 
$n$ leads to a decrease in both bias and MSE, validating the asymptotic property of the estimator.
\begin{table}[H]
\centering
\caption{Bias and MSE of $\hat{\mathcal{{CI}_\alpha^Q}}(X,Y)$, where $X \sim \text{power-Pareto }   (C=1.5,\lambda_1=0.75,\lambda_2=0.25)$ , $Y \sim \text{Exponential }(\theta=2)$.}
\begin{tabular}{@{}p{2.5cm} p{2.5cm} p{2.5cm} p{2.5cm} p{2.5cm} p{2.5cm}@{}} 
	\toprule
$\alpha$ & $\mathcal{CI_\alpha^{Q}}(X,Y)$ & $n$ &  $\hat{\mathcal{CI_\alpha^{Q}}}(X,Y)$ & Absolute bias & MSE \\ \midrule
0.25 & 0.0391 &	50 & 0.0455 & 0.0065 & 0.00179 \\
& &	100 &  0.0453 & 0.0063 & 0.00091 \\
& &	200 &  0.0437 & 0.0046 & 0.00046 \\
& &	300 & 0.0421& 0.0030 & 0.00028\\
& &  500 & 0.0416 & 0.0025 & 0.00016\\
& &	1000 & 0.0409 & 0.0018 & 0.00008\\
\hline
0.5 & 0.0373 &	50 & 0.0403 & 0.0029 & 0.00034 \\
& &	100 &  0.0392 & 0.0017 & 0.00016 \\
& &	200 &  0.0384 & 0.0010 & 0.00008\\
& &	300 & 0.0380 & 0.0007 & 0.00005\\
& &  500 & 0.0378 & 0.0005 & 0.00003\\
& &	1000 & 0.0376 & 0.0003 & 0.00002\\
\hline
0.75 & 0.07023 &	50 & 0.0721 & 0.0019 & 0.00072 \\
& &	100 &  0.0716 & 0.0013 & 0.00036 \\
& &	200 &  0.0711 & 0.0008 & 0.00019 \\
& &	300 & 0.0708 & 0.0005 & 0.00012\\
& &  500 & 0.0706 & 0.0003 & 0.00007\\
& &	1000 & 0.0705 & 0.0002 & 0.00003\\
	\bottomrule
\end{tabular}
\label{Sim 1}
\end{table}
The performance of the estimator can be validated by generating pairs of random samples for various sample sizes by taking the true model $F$ as the Govindarajulu distribution with QF $Q_X(u) = \theta +\sigma((\beta+1)u^{\beta}-\beta u^{\beta+1}),\hspace{0.2cm} \theta \in (-\infty,+\infty), \sigma, \beta > 0, \hspace{0.2cm} 0 \leq u \leq 1$, which is a quantile model and the assigned distribution $G$ as the standard uniform distribution. The bias and MSE of $\hat{\mathcal{{CI}_\alpha^Q}}(X,Y)$ relative to the true value are given in Tables ~\ref{Sim 2}.
\begin{table}[H]
\centering
\caption{Bias and MSE of $\hat{\mathcal{{CI}_\alpha^Q}}(X,Y)$, where $X \sim \text{Govindarajulu}(\theta=0.15,\sigma=0.75,\beta=3)$ , $Y \sim \text{Uniform(0,1)}.$}
\begin{tabular}{@{}p{2.5cm} p{2.5cm} p{2.5cm} p{2.5cm} p{2.5cm} p{2.5cm}@{}} 
	\toprule
$\alpha$ & $\mathcal{CI_\alpha^{Q}}(X,Y)$ & $n$ &  $\hat{\mathcal{CI_\alpha^{Q}}}(X,Y)$ & Absolute bias & MSE \\ \midrule
0.25 & 0.2149 &	50 & 0.2399 & 0.0250 & 0.02835 \\
& &	100 &  0.2278 & 0.0129 & 0.01315 \\
& &	200 &  0.2231 & 0.0082 & 0.00577 \\
& &	300 & 0.2199 & 0.0051 & 0.00341\\
& & 500 & 0.2180 & 0.0031 & 0.00209\\
& &	1000 & 0.2170 & 0.0022 & 0.00100\\
\hline
0.5 & 0.1810 &	50 & 0.1759 & 0.0051 & 0.00331 \\
& &	100 &  0.1777 & 0.0033 & 0.00197 \\
& &	200 &  0.1788 & 0.0022 & 0.00099\\
& &	300 & 0.1800 & 0.0011 & 0.00058\\
& &  500 & 0.1811 & 0.0009 & 0.00040\\
& &	1000 & 0.1811 & 0.0006 & 0.00019\\
\hline
0.75 & 0.2184 &	50 & 0.2110 & 0.0074 & 0.00255 \\
& &	100 &  0.2139 & 0.0044 & 0.00150 \\
& &	200 &  0.2158 & 0.0026 & 0.00075 \\
& &	300 & 0.2170 & 0.0014 & 0.00044\\
& &  500 & 0.2181 & 0.0002 & 0.00031\\
& &	1000 & 0.2182 & 0.0001 & 0.00014\\
	\bottomrule
\end{tabular}
\label{Sim 2}
\end{table}

\section*{Conflict of interest statement}
The corresponding author, on behalf of all authors, asserts that there are no conflicts of interest.

\section*{Acknowledgements}
 The first author gratefully acknowledges the financial support provided by the Department of Science and Technology, Government of India, under the INSPIRE Fellowship scheme (Code No: IF220243). 

 \section*{Declaration of generative AI}
 During the preparation of this work, the authors used ChatGPT (OpenAI) to improve the clarity and language of the manuscript. The authors reviewed and edited the content as needed and take full responsibility for the content of the published article.

\bibliographystyle{apalike}
\bibliography{ref} 

\end{document}